\documentclass[onefignum]{siamart220329}
\usepackage{bm}
\usepackage{listings,amsmath,amssymb,graphicx,url,array,placeins,enumitem}
\usepackage{algorithm, setspace}
\usepackage{algpseudocode}
\usepackage{parskip}
\usepackage[toc,page]{appendix}
\usepackage[inner=3cm,outer=3cm,top=3cm,bottom=3cm]{geometry}
\usepackage{soul}

\title{Deflation Techniques for Finding Multiple Local Minima of a Nonlinear Least Squares Problem}
\headers{Deflation Techniques for Nonlinear Least Squares}{Alban Bloor Riley, Marcus Webb, Michael L.~Baker}

\author{Alban Bloor Riley\footnote{Corresponding author.} \thanks{Department of Mathematics, The University of Manchester, Manchester M13 9PL, United Kingdom, \texttt{alban.bloorriley@manchester.ac.uk}, \texttt{marcus.webb@manchester.ac.uk}} \and Marcus Webb\footnotemark[2] \and Michael L.~Baker\thanks{The University of Manchester, Department of Chemistry, Manchester M13 9PL, United Kingdom, The University of Manchester at Harwell, Diamond Light Source, Harwell Campus, OX11 0DE, UK, \texttt{michael.baker@manchester.ac.uk}}}
\date{July 2024}

\begin{document}
\maketitle
\begin{abstract}
  In this paper we generalize the technique of deflation to define two new  methods to systematically find many local minima of a nonlinear least squares problem. The methods are based on the Gauss--Newton algorithm, and as such do not require the calculation of a Hessian matrix. They also require fewer deflations than for applying the deflated Newton method on the first order optimality conditions, as the latter finds all stationary points, not just local minima. One application of interest covered in this paper is the inverse eigenvalue problem (IEP) associated with the modelling of spectroscopic data of relevance to the physical and chemical sciences. Open source MATLAB code is provided at \url{https://github.com/AlbanBloorRiley/DeflatedGaussNewton}.
\end{abstract}
\begin{AMS}
    90C53, 65K05, 65K10, 90C26, 49M15
\end{AMS}
\begin{keywords}
Gauss--Newton method, deflation, nonlinear least squares
\end{keywords}

\section{Introduction} 
Nonlinear optimization problems often have many local minima. In many cases, finding an arbitrary one of them is not sufficient for the underlying application, and it is a global minimum that is sought, or several local minima that are all of interest. In this paper we generalize the technique of deflation to systematically find multiple minima of a nonlinear least squares problem. 

Deflation techniques were first analysed by Wilkinson in 1963 to find multiple roots of  polynomial equations by dividing the polynomial by the linear factor corresponding to the computed roots \cite{wilkinson_rounding_1963}. He demonstrated that the computed roots could become inaccurate if the roots are not found in ascending order of size. The method was later extended to systems of nonlinear algebraic equations in 1971 by Brown and Gearhart \cite{brown_deflation_1971}. Unfortunately they also found that deflation was often unstable and could lead to divergence.

Decades later, in 2015 Farrell et al.~discovered a simple improvement to the deflation process \cite{farrell_deflation_2020,farrell_deflation_2015}: rather than multiplying the rootfinding problem by $|y - x|^{-1}$, for example, where $y$ is a previously computed root, one can instead multiply by $1 + |y-x|^{-1}$. This general idea of ``shifting'' the deflation operation means that away from the previously computed root, the deflated problem is well-approximated by the original problem, stabilizing the process of computing subsequent roots.

How to apply deflation techniques to optimization problems is an interesting question. The obvious approach is to apply deflation to the rootfinding problem associated with the first order optimality conditions \cite{papadopoulos_computing_2021, friedland_formulation_1987}. In this paper, however, we apply deflation techniques directly to the Gauss--Newton nonlinear least squares optimization algorithm. To do this, we define conditions on deflation operators specifically for optimization algorithms, leveraging a finer understanding of what deflation does to the behaviour of rootfinding algorithms. These new methods do not require the (possibly expensive) calculation of Hessian matrices nor do they converge to local maxima or saddle points.

In Section \ref{sec. Deflation Intro} we discuss the theory of deflation for rootfinding problems. In Section \ref{sec. NLLS} we discuss the deflated Newton method, used in \cite{farrell_deflation_2020} and \cite{papadopoulos_computing_2021} to find multiple solutions to nonlinear PDEs, then introduce and prove properties of two new deflated methods, the ``good'' and ``bad'' deflated Gauss--Newton methods. Finally, in Section \ref{sec: Application} we conduct numerical experiments to compare our new methods to the existing deflated Newton method and the MultiStart method in MATLAB \cite{inc_global_2024}. Open source MATLAB code for all experiments is provided at \url{https://github.com/AlbanBloorRiley/DeflatedGaussNewton}.

\section{Deflation techniques for rootfinding}\label{sec. Deflation Intro}
Given a smooth function $r:\mathbb R^m\rightarrow\mathbb R^m$, and its Jacobian $J_r:\mathbb R^m\rightarrow\mathbb R^{m\times m}$  the system of equations $ r( x) = 0$  can be solved using the Newton method given in Algorithm \ref{alg. Newton}. If we wish to find multiple solutions of $r$ we may use the deflated Newton method in Algorithm \ref{alg. Deflated Newton}. A deflation operator should satisfy the following to be effective \cite{farrell_deflation_2020}:

\begin{definition}{(\cite{farrell_deflation_2020})}\label{def. Deflation Operator}
Let $r$ be as above and let $y_1,\dots,y_n \in \mathbb R^m$. We say that $\mu = \mu(\cdot ;y_1,\ldots,y_n):\mathbb R^m \setminus \{y_1,\ldots,y_n\} \to \mathbb R$ is a deflation operator for $r$ if it satisfies the following:
\begin{enumerate}
    \item The function $\nabla\mu(x;y_1,\dots,y_n)$ exists and is continuous for all $x \in \mathbb{R}^m \setminus \{y_1,\ldots,y_n\}$.
    \item $\displaystyle \liminf_{ x \rightarrow y_i} ||\mu(x;y_1,\dots,y_n)r( x)||_2 >0 $.
    \item If $r(x)=0$ and $x\notin \{y_1,\dots,y_n\}$ then $\mu(x;y_1,\dots,y_n) r(x) =0$.
    \item  If $r(x) \neq 0$ then $\mu(x;y_1,\dots,y_n)r(x)\neq0$.
\end{enumerate}
\end{definition}
\begin{algorithm}\caption{The Newton method for nonlinear equations}\label{alg. Newton}
\setstretch{1.25}
\begin{algorithmic}
\State $ x^0 =$ initial guess
\While{$||r(x^k)||_2<$ tol}
\State Solve $J_r(x^k)p^k = -r(x^k)$
\State $ x^{k+1} =  x^k +  p^k$
\EndWhile
\end{algorithmic}
\end{algorithm}

\begin{algorithm}\caption{The deflated Newton method for nonlinear equations \cite{farrell_deflation_2020}}\label{alg. Deflated Newton}
\setstretch{1.25}
\begin{algorithmic}
\State $ x^0 =$ initial guess
\While{$||r(x^k)||_2<$ tol}
\State Solve $J_r(x^k)p^k=- r(x^k)$
\State$\displaystyle \beta = 1-\langle\mu(x^k)^{-1}{\nabla\mu(x^k)},p^k\rangle$
\State $\displaystyle x^{k+1} =  x^k + \beta^{-1} p^k$
\EndWhile
\end{algorithmic}
\end{algorithm}

\begin{theorem}[Farrell et al.~\cite{farrell_deflation_2020}]\label{Thrm: Deflated Newton}
    A step of Algorithm \ref{alg. Deflated Newton} is equivalent to a step of Algorithm \ref{alg. Newton} applied to the system $\mu(x)r(x)$.
\end{theorem}
\begin{proof}
    A short proof can be found in Appendix A.
\end{proof}

Notice that in Algorithm  \ref{alg. Deflated Newton} the value $ \beta = 1-\langle\mu^{-1}{\nabla\mu},p^k\rangle$ is a scalar, and so the only effect that deflation has on the Newton method is that each step is just a scalar multiple of the undeflated Newton step $p^k$. This means that the effect of deflation can be fully described by the value of $\beta$, and in fact the behaviour can be split up into three distinct cases, which are displayed graphically in Figure \ref{fig: Simple Beta Contour} for Himmelblau's function (more specifically, the residual of),
\begin{equation}\label{eqn:Himmelblau}
    r(x,y) = (x^{2}+y-11, x+y^{2}-7)^T.
\end{equation}
\begin{enumerate}
    \item When $\beta<0$ (alternatively $\langle\mu^{-1}{\nabla\mu},p^k\rangle > 1$), $p^k$ is $\mu$-ascending and the deflation update reverses the direction of the step.
    \item When $0<\beta<1$  (or $0< \langle\mu^{-1}{\nabla\mu},p^k\rangle <1$), $p^k$ is moderately $\mu$-ascending and deflation increases the step length. This can be interpreted as destabilising potential convergence to a deflated point.
    \item Finally, when $\beta > 1$ ($\langle\mu^{-1}{\nabla\mu}, p^k\rangle<0$), $p^k$ is $\mu$-descending and deflation decreases the step length. If $\mu$ is relatively flat away from deflated points then this decrease is minimal there.
\end{enumerate}
\begin{figure}[!ht]
    \centering
    \includegraphics[width =0.8\textwidth]{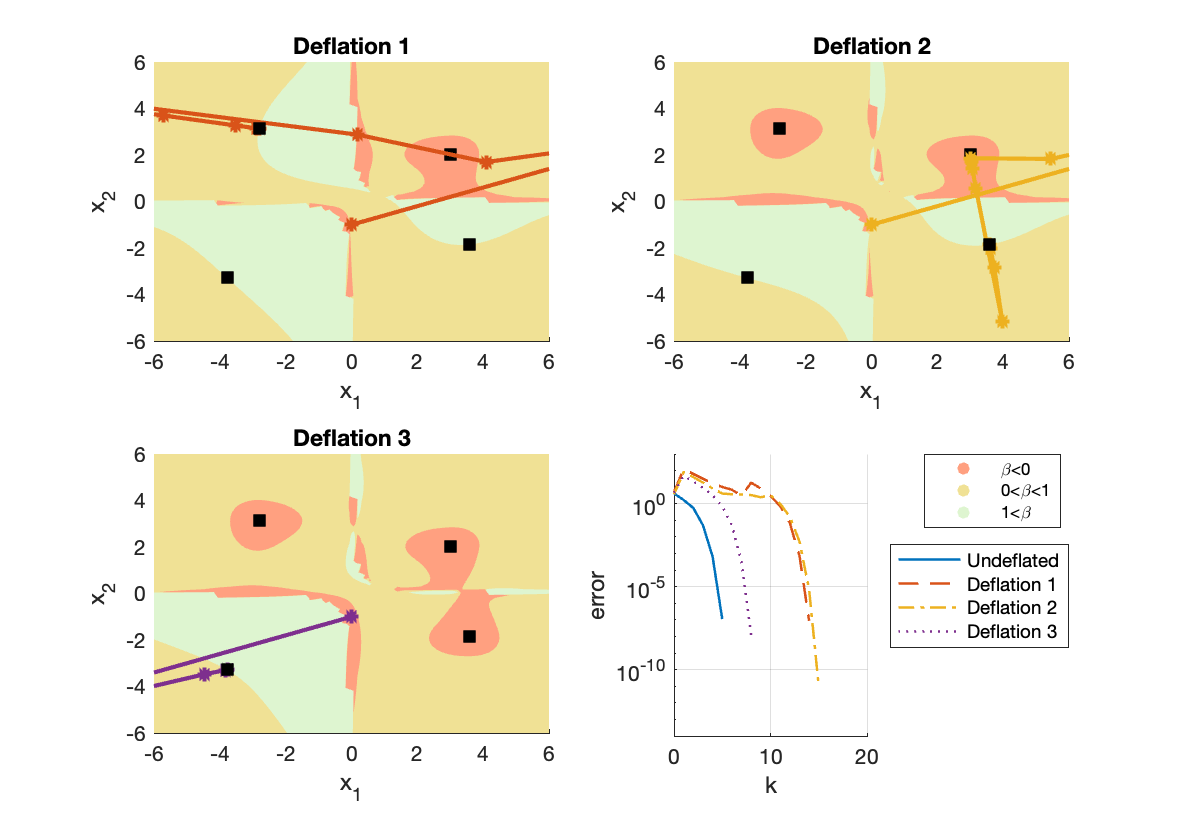}
    \caption{Multiple solutions to Himmelblau's function \cite{himmelblau_applied_1972}, defined above, computed using the deflated Newton method. Bottom Right: Shows the convergence rate of the method to all 4 roots. The other plots show the contours of $\beta$ after 1, 2 and 3 deflations respectively. Notice that unfound solutions lie on the contour $\beta = 1$ and deflated solutions lie in a region where $\beta < 0$.}
    \label{fig: Simple Beta Contour}
\end{figure}

\subsection{Deflation operators}\label{Section: Deflation Operators}
Early use of deflation was to find multiple roots of polynomial equations \cite{wilkinson_rounding_1963}. Given a polynomial, $p_n(x)$, and a root $y$ that has been found by some iterative method, then the quotient polynomial 
\begin{equation}
    q_n(x) = \frac{p_n(x)}{y-x},
\end{equation}
is formed and a new root can be found by applying the same iterative rootfinding method to $q_n(x)$, a process which can be repeated. Wilkinson's careful analysis revealed numerical instabilities to this approach that depend on the order in which the roots are found, with ascending order in magnitude being the most stable order.

Brown and Gearhart \cite{brown_deflation_1971} introduced the deflation operator
\begin{equation}
    \mu(x;y) = \frac{1}{|| y-x||} 
\end{equation}
for some vector norm, along with other matrix-based operators. In principle, these deflation operators can be applied to any rootfinding problem for $r : \mathbb{R}^m \to \mathbb{R}^m$. Unfortunately, they observed similar instabilities in this approach.

Farrell et al.~modified the deflation operator by adding a shift, $\sigma$ \cite{farrell_deflation_2020}. This results in the deflated system resembling the undeflated system when far enough away from any deflated points. The value of $\sigma$, is of course a parameter that needs to be defined. In this paper we will set $\sigma = 1$. However, we note that other choices are valid and can in theory lead to better results \cite{farrell_deflation_2020}. There is also the choice of norm,  for simplicity we will use the  2-norm, but it is possible to use different norms (see Subsection \ref{subsec FE}). It is also often advantageous to raise the norm to some power, $\theta$, intended to overcome a higher multiplicity in the roots. In practice, $\theta = 2$ is effective. This all leads to the deflation operator that we will use in this paper:
\begin{equation}\label{eqn. Deflation Operator}
    \mu( x;y) = \frac{1}{|| y-x||^\theta_2} + \sigma,
\end{equation} 
which satisfies both the root deflation Definition \eqref{def. Deflation Operator} and the optimization deflation Definition \eqref{def. Optimization Deflation Operator} defined in Section \ref{sec. NLLS}.

\subsubsection{Multiple deflations}
After a root, say $y_1$ has been found to a system of equations $r(x)$ we can deflate out this found root by defining the deflated system $\mu( x;y_1)r( x)$ and running the same rootfinding method on this new problem. This can be iterated to obtain $\mu( x;y_n)\cdots \mu( x;y_1)r( x)$ leading to the `multi-shift' deflation operator:
\begin{equation}\label{eqn. Multi Shift Operator}
    \mu( x; y_1,\ldots, y_n) = \mu( x;y_1)\dots \mu( x;y_n) = \left(\sigma + \frac{1}{|| x -y_1||_2^\theta}\right) \dots \left(\sigma +\frac{1}{|| x -y_n||_2^\theta}\right)
\end{equation}
An alternative to consider is a `single shift' deflation operator:
\begin{equation}
    \mu( x; y_1,\ldots, y_n) = \sigma + \frac{1}{|| x -y_1||_2^\theta} \cdots \frac{1}{|| x -y_n||_2^\theta}.
\end{equation}

Note that sometimes, $y_i = y_j$ for distinct $i$ and $j$. This can happen if the root of $r$ at $y_i$ is of multiplicity too great to be cancelled out by $\mu$ \cite{brown_deflation_1971}.

\subsubsection{A novel deflation operator}
 The choice of $\theta$ was motivated by the expected multiplicity of the roots to be found. However, in many cases, this information will not be known, and a somewhat arbitrary choice needs to be made. We now provide an alternative operator that aims to incorporate all possible values of $\theta$. Consider the expansion,
\begin{equation}
    \exp\left( \frac{1}{|| y-x||_2}\right) = 1 + \frac{1}{|| y-x||_2} + \frac1 2 \frac{1}{|| y-x||_2^2} +\cdots.
\end{equation}
 This expression incorporates all powers $\theta$, so we  define a new  deflation operator $\mu( x; y) = \exp( \frac{1}{|| y-x||_2})$. This deflation operator has not been investigated before. 

 This deflation operator is natural because deflation only involves $\mu^{-1}\nabla\mu$, which is relatively simple in this case:
\begin{equation}
\mu(x)^{-1}\nabla\mu(x) = \frac{y-x}{||y-x||_2^3}.
\end{equation}\label{exp deflation operator}

\section{Deflation techniques for nonlinear least squares problems}\label{sec. NLLS}
A nonlinear least squares problem is that of minimising the function $f : \mathbb{R}^\ell \to \mathbb{R}$ given by
\begin{equation}\label{NLLS equation}
    f( x) = \frac12||r( x)||^2_2,
\end{equation}
where $r:\mathbb{R}^\ell \to \mathbb{R}^m$. Such problems arise naturally when fitting $m$ data-points $r_i$ with $\ell$ parameters $x_j$ in the presence of Gaussian noise \cite{nocedal_numerical_2006}.

In this section we discuss three different deflated methods for finding multiple minimizers of a given $f$. The first method we will discuss is the deflated Newton method given here as Algorithm \ref{alg. Deflated NLLS Newton}.
\begin{algorithm}\caption{The deflated Newton method for optimization}\label{alg. Deflated NLLS Newton}
\setstretch{1.25}
\begin{algorithmic}
\State $ x^0 =$ initial guess
\While{$||\nabla f(x^k)||_2<$ tol}
\State Solve $H_f(x^k)p^k=-\nabla f(x^k)$
\State$\displaystyle \beta = 1-\langle\mu(x^k)^{-1}{\nabla\mu(x^k)},p^k\rangle$
\State $\displaystyle x^{k+1} =  x^k + \beta^{-1} p^k$
\EndWhile
\end{algorithmic}
\end{algorithm} 

One of the main downsides to the Newton method for optimization, is that it requires calculating second derivatives of the objective function, which can be a relatively expensive operation, or indeed simply not available. Another disadvantage compared to Gauss--Newton methods is that the method can converge to any given stationary points of $f(x)$, not just to the local minima, so often computation is wasted on these unwanted stationary points. To remedy these two disadvantages, we introduce the ``good'' and the ``bad'' Deflated Gauss--Newton methods, presented here as Algorithms \ref{alg. Good Deflated Gauss--Newton} and \ref{alg. Bad Deflated Gauss--Newton}.

\begin{definition}\label{def:levelset}
For a function $g : \mathbb{R}^d \to \mathbb{R} \cup \{+\infty \}$ and $c \in \mathbb{R}$, the strict $c$-superlevel set of $g$ is the set
\begin{equation*}
    L_c(g) := \left\{ x \, | \, g(x) > c \right\}.
\end{equation*}
\end{definition}
For the rest of this paper we will just use the term level set when talking about a strict superlevel set.
In the rest of this section we will discuss the motivation and convergence analysis for the three methods.  However, we first need a new definition of deflation operator for nonlinear least square problems. It is helpful to define the function $\eta(x) = \log(\mu(x))$.
\begin{definition}\label{def. Optimization Deflation Operator}
For an iterative minimization method $x^{k+1} = x^k + p(x^k)$ and a set of deflated points $y_1,\dots,y_n$, we say that $\mu( x; y_1,\ldots, y_n):\mathbb R^\ell \setminus \{y_1,\ldots,y_n\} \rightarrow \mathbb{R}_{>0}$ is a deflation operator if it satisfies:
\begin{enumerate}
    \item The function $\nabla \eta(x) = \mu(x)^{-1} \nabla \mu(x)$ exists and is continuous for all $x \in \mathbb{R}^\ell \setminus \{y_1,\ldots,y_n\}$.
    \item If $p(y_i) = 0$ then $\liminf_{ x \rightarrow y_i} \langle  \nabla\eta(x), p(x)\rangle >1$.
    \item There exists a number $c^* > 0$ such that for any $c > c^*$ the level set $L_{c}(\mu)$ consists of $n$ nonempty convex sets, each containing one of $y_1,\ldots,y_n$.
\end{enumerate}
\end{definition}

Note that this definition differs from  Definition \ref{def. Deflation Operator} given earlier for systems of equations. In this definition neither $f$ nor $r$ is referred to directly, only the optimization step $p$. 

This definition also does not explicitly exclude the possibility that deflation can create spurious solutions. We believe that this property depends on how the deflation operator is applied (we define two ways below). In practice we don't observe spurious solutions, but a theoretical understanding is lacking at this point in time.

\begin{algorithm}\caption{The ``good" deflated Gauss--Newton method}\label{alg. Good Deflated Gauss--Newton}
\setstretch{1.25}
\begin{algorithmic}
\State $ x^0 =$ initial guess
\State $\epsilon \in [0,1]$ is as described in \eqref{epsilon}
\While{$||p^k||_2<$ tol}
\State $p^k = {\operatorname{argmin}_p} \| r(x^k) + J_r(x^k) \, p \|_2$
\If {$\langle p^k, \nabla \eta(x^k)\rangle >\epsilon$}
\State$\displaystyle \beta = 1- \langle \nabla\eta(x^k), p^k \rangle$
\State $\displaystyle  x^k =  x^{k-1} + \beta^{-1}p^k$
\Else
\State $\displaystyle  x^{k+1} = x^k + \alpha \, p^k$, where $\alpha$ is determined by line search on $f$.
\EndIf
\EndWhile
\end{algorithmic}
\end{algorithm}

\begin{algorithm}\caption{The ``bad" deflated Gauss--Newton method}\label{alg. Bad Deflated Gauss--Newton}
\setstretch{1.25}
\begin{algorithmic}
\State $ x^0 =$ initial guess
\State $\epsilon \in [0,1]$ is as described in \eqref{epsilon}
\While{$||p^k||_2<$ tol}
\State $p^k = \operatorname{argmin}_p \| r(x^k) + J_r(x^k) \, p \|_2$
\If{$\langle p^k, \nabla \eta(x^k)\rangle >\epsilon$}
\State $\hat{p}^k = \beta_1 \, p^k + \beta_2\, \left(J_r(x^k)^T J_r(x^k)\right)^{-1} {\nabla\eta(x^k)} $ where $\beta_1$ and $\beta_2$ are defined in equation \eqref{eqn:beta1beta2}.
\State $x^{k+1} = x^k + \hat{p}^k$
\Else
\State $x^{k+1} = x^k + \alpha \, p^k$, where $\alpha$ is determined by line search on $f$.
\EndIf
\EndWhile
\end{algorithmic}
\end{algorithm}

\subsection{Newton's method for optimization}
Algorithm \ref{alg. Deflated NLLS Newton} is in fact still a rootfinding method and is  the same as Algorithm \ref{alg. Deflated Newton} applied to $\nabla f(x)$ the gradient of $f(x)$, thus finding the stationary points of $f$. The algorithm requires the calculation of the gradient, $\nabla f \in \mathbb R^{\ell}$, and Hessian, $H_f \in\mathbb R^{\ell\times\ell}$, of $f$. The gradient is given by the formula 
\begin{equation}\label{eqn:gradf}
    \nabla f(x) = J_r( x)^Tr( x)
\end{equation}
note that we will assume for the rest of this paper $J_r(x^k)$ is full rank for all $k$. The Hessian by 
\begin{equation}
    H_f = J_r( x)^TJ_r( x) + \sum_{i=1}^m r_i( x) H_{r_i}( x)
\end{equation}
where $r_i( x)$ represents the $i$th component of $r$ and $H_{r_i}( x)$ represents the Hessian of $r_i$.

Since the Newton algorithm for optimization can equivalently be thought of as the rootfinding Newton method (Algorithm \ref{alg. Newton}) applied to $\nabla f$, it will converge to stationary points of $f$ given a sufficiently good initial guess. 

The work of Papadopoulos et al.~in  \cite{papadopoulos_computing_2021} has shown that this deflated method can be effective at solving large-scale topology optimization problems.

\subsubsection{Convergence}
The Newton method for optimization is not guaranteed to be a globally convergent method, this is due to the fact that $H_f$ may not be positive definite far from a minimum. There are results, however, about the local convergence properties in the neighbourhood around a minimum.


\begin{theorem}[Theorem 2.4 in \cite{nocedal_numerical_2006}]
    Let $f(x)$ be twice differentiable, and $H_f(x)$ Lipshitz continuous in a neighbourhood around a point, say $x^*$. Let us further assume  $\nabla f(x^*) = 0$, and $H_f(x^*)$ is positive definite (that is $x^*$ is a strict local minimiser i.e. the space is locally convex). Then if the starting point, $x^0$, is sufficiently close to $x^*$  the method will converge quadratically to $x^*$, and the sequence of norms of the gradients will also converge quadratically to $0$.
\end{theorem}

The Deflated Newton method will not converge to previously deflated points, given that large enough exponent, $\theta$, is chosen, or after a sufficient number of deflations \cite{brown_deflation_1971}. 

\FloatBarrier

\subsection{Gauss--Newton methods}
First we will look at the undeflated method. This is an algorithm specifically designed for solving the nonlinear least squares problem, it is given here as Algorithm \ref{alg. Gauss--Newton}.  Importantly it does not require the calculation of the Hessian of $f$, nor does it converge to all stationary points of  $f$, just the minima (or maxima) that are of interest.  

\begin{algorithm}[htbp]\caption{The Gauss--Newton method}\label{alg. Gauss--Newton}
\setstretch{1.25}
\begin{algorithmic}
\State $ x^0 =$ initial guess
\While{$\|p^k\|_2<$ tol}
\State $p^k = {\operatorname{argmin}_p} \| r(x^k) + J_r(x^k) \, p \|_2$
\State $ x^{k+1} =  x^{k} +\alpha \, p^k$, where $\alpha$ is determined by line search on $f$.
\EndWhile
\end{algorithmic}
\end{algorithm}
Interestingly in most textbooks the calculation of $p^k$ in the algorithm is given by solving:
\begin{equation}
    (J_r(x^k)^TJ_r(x^k)) p^k = - J_r(x^k)^Tr(x^k).
\end{equation}
This might due to the fact that the method is frequently described as a truncated Newton method; since $J_r^TJ_r$ can be thought of as the approximation to $H_{f}$ given by neglecting the second order terms, and $J_rr = \nabla f$. Note that this observation motivates one approach to deflating the method, covered in Section \ref{sec. Good GN}. A more formal motivation of the method however comes from  looking at the first two terms of the Taylor series expansion of $r$ at point $ x^k$: 
\begin{equation}
    r( x) \approx r( x^k) + J_r( x^k)( x -  x^k).
\end{equation}
We then try to find the $ -p= ( x^{k} -  x^{k+1}) $ that minimizes this surrogate linear least squares problem:
\begin{equation}
    {\operatorname{argmin}_p} || r( x^k) + J_r( x^k) p||_2^2.
\end{equation}
Clearly in the square case the minimum occurs when $ p  =- J_r( x^k)^{-1}r( x^k)$ (assuming that $J_r^{-1}$ exists). This result generalizes in the rectangular case, such that the minimum can be given by using the pseudoinverse: $ p = -J_r( x^k)^+r( x^k) $, although of course it is in general best to avoid explicitly calculating the pseudoinverse.  Note  that if $J_r$ is  left invertible --- which is always the case when $J_r$ has full rank --- we can rewrite the pseudoinverse as $J_r^+ = (J_r^TJ_r)^{-1} J_r^T$. Thus recovering the formula above.

Note that the use of the pseudoinverse shows that the method can  be thought of as a generalization of the rootfinding  Newton method applied to $r$. This provides the motivation for the other approach to applying deflation to the Gauss--Newton method, covered in Section \ref{Sec. Bad GN}.

\subsection{Convergence of the undeflated Gauss--Newton method}\label{sec. undeflated GN convergence}
Like the Newton method it is known that the Gauss--Newton method is locally convergent, under certain assumptions, but not necessarily globally convergent.

\begin{theorem}[Theorem 10.1 in \cite{nocedal_numerical_2006}]\label{thm: undeflatedGN}
Assume that $r(x)$ is Lipschitz continuously differentiable  and that $J_r(x)$  is uniformly full rank\footnote{i.e.~there exists a constant $\gamma > 0$ such that $\|J_r(x) v\| \geq \gamma \|v\|$ for all $v \in \mathbb{R}^\ell$ and all $x \in \mathbb{R}^m$.} for all $x$ in a neighbourhood of the bounded level set $\{x | f(x) \leq f(x^0) \}$. If $x^k$ are iterates generated by the Gauss--Newton method using a line search that satisfies the Wolfe conditions then 
\begin{equation}
    \lim_{k\to\infty} J_r(x^k)^Tr(x^k) = 0.
\end{equation}
\end{theorem}
There are, however, no guarantees if $J_r$ is not uniformly fully rank \cite{nocedal_numerical_2006}.


\subsection{The ``good" deflated Gauss--Newton method}\label{sec. Good GN}
Our first approach to deflating the Gauss--Newton method is motivated by the fact that it can be thought of as a truncated Newton method, and thus defined by the step:
\begin{equation}
     x^{k+1} =  x^{k} - (J_r^TJ_r)^{-1} J_r^Tr.
\end{equation}

To derive a deflated method, we can apply similar logic that derives Gauss--Newton from optimization Newton by dropping Hessian terms. First we need to calculate the derivative of the deflated gradient:
\begin{equation}
    \nabla(\mu J_r^Tr) = J_r^Tr{\nabla\mu}^T + \mu J_r^TJ_r + \mu \sum_{i=1}^mr_i H_{r_i}.
\end{equation}
Substituting this into the formula for the step of the optimization Newton algorithm
\begin{equation}
    \left(\frac{1}{\mu}J_r^Tr{\nabla\mu}^T + J_r^TJ_r  + \sum_{i=1}^mr_i H_{r_i}\right)^{-1} J_r^Tr.
\end{equation}
Since the  Gauss--Newton method is given by neglecting the second order terms in Newton's method, we do  the same here with the deflated components to get the deflated step:
\begin{equation}
    \left(\frac{1}{\mu}J_r^Tr{\nabla\mu}^T +  J_r^TJ_r\right)^{-1} J_r^Tr.
\end{equation}
The same logic as in the proof of Theorem \ref{Thrm: Deflated Newton} gives the same simplification:
\begin{equation}
    x^{k+1} = x^k + \beta^{-1} p^k
\end{equation}
where $p^k$ is the Gauss--Newton step and $\beta = 1- \langle\nabla\eta,p^k\rangle$. 

\subsubsection{Convergence and non-convergence}

In order to talk about the convergence and indeed  non-convergence of this method it is again useful to look at the three different cases for $\beta$, and what the effect of deflation is in each of these cases. In fact we propose adapting the method depending on which case the current iterate satisfies. 

\begin{itemize}
    \item When $1 <\beta$ (alternatively $\langle \nabla \eta, p\rangle<0$), as stated before, the step $p$ is already in a descent direction for $\mu$, and the effect of deflation reduces the step size. In this case we suggest using an undeflated step with a line search on the undeflated objective function, thus keeping the convergence guarantees of the undeflated method covered in Section \ref{sec. undeflated GN convergence}. This approach could be thought of as  incorporating the method of `purification' of the found minima; a  technique introduced by Wilkinson in \cite{wilkinson_rounding_1963} to more accurately calculate all roots.  
    \item When $\beta <1$ ($0 < \langle \nabla \eta, p\rangle$), the step is in an ascent direction for $\mu$. However, when $\beta<0$ ($1<\langle \nabla \eta, p\rangle$) the effect of deflation is to reverse the direction of the step and so is a descent direction for $\mu$.
\end{itemize}

We know that all unfound solutions lie on the boundary $\beta=1$ ($\langle \nabla \eta, p\rangle=0$), since on this contour $p = 0$, and $\nabla \eta$ is finite by Definition \ref{def. Optimization Deflation Operator} part (1). Ideally, a line search would be applied within a neighbourhood around each minimum to guarantee convergence. To achieve this we suggest relaxing the condition on $\beta$  slightly so that the undeflated Gauss--Newton step with line search is used when 
\begin{equation}
    \beta >1-\epsilon,\quad \text{equivalently}\quad \langle \nabla \eta, p\rangle<\epsilon \label{epsilon}
\end{equation}
for small $\epsilon \in [0,1]$. The effect of different values of  $\epsilon$ can be seen in Figure \ref{fig: epsilon contours} where the green region, where undeflated line search Gauss--Newton is used, is significantly expanded for a positive value of $\epsilon$. In our experiments we don't observe much impact of a small $\varepsilon$ against $\varepsilon = 0$ on the rate of convergence of the ``good'' deflated Gauss--Newton algorithm.

\begin{figure}[!ht]
    \centering
    \includegraphics[width =0.8\textwidth]{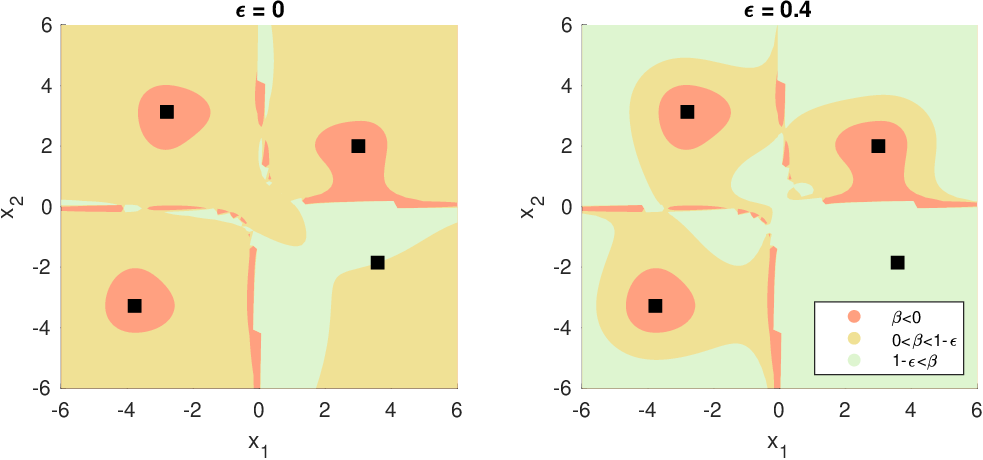}
    \caption{The effect of setting $\epsilon > 0$ on the pertinent contours of $\beta$, for the Himmelblau function \eqref{eqn:Himmelblau} after three deflations. Notice that fourth solution lies on the contour $\beta = 1$.}
    \label{fig: epsilon contours}
\end{figure}

Thus the ``good" deflated Gauss--Newton method, Algorithm \ref{alg. Good Deflated Gauss--Newton}, uses the undeflated Gauss--Newton step with a line search on $f$ when $\langle \nabla \eta, p\rangle \leq \epsilon$, and uses the deflated step, utilising the scalar update, with no line search applied when $\langle \nabla \eta, p\rangle>\epsilon$. 

It is important to show that these deflated methods do not converge to any point that has been deflated.

\begin{theorem}\label{Thrm: Good GN NC}
        Let $y_1,\dots,y_n \in \mathbb R^\ell$ be a set of deflated points and $\mu(x;y_1,\dots,y_n)$ be a deflation operator as in Definition  \ref{def. Optimization Deflation Operator}. Let $\{x^k\}_{k=0}^\infty$ be the iterates of the ``good'' deflated Gauss--Newton method (Algorithm \ref{alg. Good Deflated Gauss--Newton}). Assume that $r$ is Lipschitz continuously differentiable, $J_r$ is uniformly full rank in a neighbourhood of $\{x^k\}_{k=0}^\infty$ and assume all steps are well-defined. Then $x^k$ will not converge to any of the deflated points $y_1,\dots,y_n$.
\end{theorem}
\begin{proof}
Assume for contradiction  that $x^k $ converges to $y \in \{y_1,\dots,y_n\}$. 

Recall that the deflated step, which we will denote by $\hat p^k$, is
\begin{equation}
        \hat p^k = \frac{p^k}{\beta(x^k)}= \frac{p^k}{1-\langle\nabla\eta,p^k\rangle}.
\end{equation}
When $\beta(x^k)<1-\epsilon$ then the deflated step is taken: $x^{k+1} = x^k + \hat p^k$, and when $\beta(x^k) \geq 1-\epsilon$ then the undeflated step is taken with a line search on $f$: $x^{k+1} = x^k + \alpha \, p^k$. Thus we need to investigate the convergence behaviour of both of these steps.  By the continuity of $\beta$ (as a function of $x \in \mathbb{R}^d \setminus \{y_1,y_2,\ldots,y_n\}$), the tail of the sequence $\{x^k\}_{k=0}^\infty$ is formed by taking either all undeflated steps, $\alpha \, p^k$, or all deflated steps, $\hat p^k$.

First we assume that the tail of $\{x^k\}_{k=0}^\infty$ only takes undeflated steps. Then, as we have assumed the sequence (with a line search) converges, by Theorem \ref{thm: undeflatedGN}, we have $J_r(x^k)^T r(x^k) \to 0$. Since $J_r$ is uniformly full rank, this implies that $p^k \rightarrow 0$. By Definition  \ref{def. Optimization Deflation Operator}, part (2),
\begin{equation}\label{eqn: condition 2}
    \liminf_{ k \rightarrow \infty} \langle  \nabla\eta(x^k), p^k\rangle >1.
\end{equation} 
By definition of the algorithm, however, this means the algorithm takes the deflated step, $\hat p^k$, a contradiction. 

The remaining case to consider is that the tail of $\{x^k\}_{k=0}^\infty$ takes only deflated steps. The convergence of the sequence implies $\lim_{k\to \infty} \hat{p}^k = 0$. Since $\hat{p}^k = \beta(x^k)^{-1} p^k$, this implies that $p^k \to 0$ or $\beta(x^k) \to -\infty$ (note that deflated steps are taken when $\beta(x^k) < 1-\epsilon$, so $\beta(x^k) \to +\infty$ is impossible). First we assume that $p^k\rightarrow0$. Then by part (2) of Definition  \ref{def. Optimization Deflation Operator} we again have the property in equation \eqref{eqn: condition 2} above. If we assume instead that $\beta(x^k) \rightarrow -\infty$ then we also satisfy \eqref{eqn: condition 2}.

The property in equation \eqref{eqn: condition 2} implies that the level set $L_{1}(\delta)$, where $\delta(x) = \langle \nabla \eta(x), p(x)\rangle$, contains the tail of $\{x^k\}_{k=0}^\infty$. Note that $L_1(\delta)$ is the red region in Figure \ref{fig: epsilon contours}. For all points $x^k$ within the region $L_{1}(\delta)$ we know that $1<\langle \nabla \eta(x^k), p^k\rangle$, and since $\beta(x^k) <0$, we have $\langle \nabla\eta(x^k), \hat p^k \rangle < 0$, and hence $\langle \nabla\mu(x^k), \hat p^k \rangle < 0$.

Now, by part (3) of Definition  \ref{def. Optimization Deflation Operator} there exists a level set $ L_{c}(\mu)$ containing $y$, and also the tail of $\{x^k\}_{k=0}^\infty$ (since the level set is a nonempty open set), such that the connected component containing $y$ is convex and the same holds for all higher level sets. Let $C$ be the closure of the particular convex subset of $ L_{\mu(x^k)}(\mu)$ that contains $y$, where $k$ is sufficiently large so that $\mu(x^k) > c$ (which is guaranteed since $\{x^k$\} converges to $y$ which is contained in the nonempty level set $ L_{c}(\mu)$). Since $x^k$ is on the boundary of the closed convex set $C$, and $\langle \nabla \mu(x^k), \hat p^k \rangle < 0$, we have that $x^{k+1} = x^k + \hat{p}^k$ lies outside of $C$ (because the hyperplane associated to $\langle \nabla\mu(x^k), \hat p^k \rangle < 0$ separates $C$). This contradicts that the tail of $\{x^k\}_{k=0}^\infty$ lies within $C$.
\end{proof}

In this section we have shown that the ``good'' deflated Gauss--Newton method is locally convergent and does not converge to deflated points, as in Definition  \ref{def. Optimization Deflation Operator}.

\subsection{The ``bad" deflated Gauss--Newton method}
\label{Sec. Bad GN}
This method results from applying the Gauss--Newton method to $\mu r$. Why is it so ``bad''? Local minima, $\tilde{x}$, of the deflated nonlinear least squares problem satisfy $J_{\mu r}^T\mu r = 0$, so are not local minima of the undeflated problem (unless $r(\tilde{x}) = 0$). However, the use of the undeflated step when $\langle p^k ,\nabla \eta (x^k) \rangle \leq \varepsilon$ ensures that the behaviour of the method in a neighbourhood of an undiscovered local minimum of $\|r\|^2$ is unaffected. Thus for sufficiently large $\epsilon$ the ``bad'' deflated Gauss--Newton iterations will not converge to $\tilde{x}$, but exactly how small $\epsilon$ could be taken is not clear. This phenomenon can be clearly seen in Figure \ref{fig: epsilon convergence comparison}. The naming convention is intended to be a light-hearted reference to that of Broyden's ``good'' and ``bad'' methods.

The derivation taking this definition to Algorithm \ref{alg. Bad Deflated Gauss--Newton} requires that we understand the Moore-Penrose pseudoinverse of $J_{\mu r}$. The Sherman--Morrison formula does not apply to non-square matrices, but fortunately we can use the work of Meyer in \cite{meyer_generalized_1973} (later generalized in \cite{guttel_sherman--morrison--woodbury_2024}) to adapt the update formula. 

We assume that $J_r$ is full rank, and thus $\nabla\eta \in R(J_r^T)$. Therefore we can utilize Theorem 5 from \cite{meyer_generalized_1973}. The possibility that $J_r$ may be rank-deficient is beyond the scope of this paper.

\begin{theorem}\label{Thrm:GN2 full rank}
    Let $A\in \mathbb R^{m\times l}$ be full rank, $u\in\mathbb R^m$ and $ v\in\mathbb R^\ell$. Then \begin{equation}
        (A+uv^T)^+u =\frac{\beta}{\omega}{A^+u} +\frac{\|Pu\|_2^2}{\omega}({A^TA)^{-1} v} 
    \end{equation}
    where $\beta = 1+\mu^{-1}{\nabla\mu}^T A^+u$, \quad $\omega= \|Pu\|_2^2 \| A^{+T}v\|_2^2 + \beta^2$, \quad and $P = I-AA^+$. 
\end{theorem} 
\begin{proof}
From \cite[p.316]{meyer_generalized_1973} we use Theorem 5:
\begin{flalign}\label{original generalized formula}
    (A+uv^T)^+ &= A^+ + \frac{1}{\beta}A^+ {A^+}^Tvu^TP - \frac{\beta}{\omega} pq^T, 
\end{flalign}
where
\begin{flalign*}
    p&= \frac{\|Pu\|_2^2}{\beta}A^+ {A^+}^T v +A^+u,\qquad
    q^T= \frac{v^TA^+{A^+}^Tv}{\beta}u^TP +v^T{A^+}.
\end{flalign*}

Similarly to the original update formula, we do not  need to store the pseudoinverse explicitly, since we only need the result of its multiplication with $u$. We thus calculate  this action:
 \begin{align}
    (A+uv^T)^+u &= A^+u + \frac{\|Pu\|_2^2}{\beta}A^+ {A^+}^Tv - \frac{\beta}{\omega} pq^Tu\nonumber\\
    \intertext{Substituting in $p$:}
    &=A^+u + \frac{\|Pu\|_2^2}{\beta}A^+ {A^+}^Tv - \frac{\beta}{\omega} \left(\frac{\|Pu\|_2^2}{\beta}A^+ {A^+}^T v +A^+u\right)q^Tu\nonumber\\
    \intertext{Expanding out the brackets and rearranging the scalar:}
    &=A^+u + \frac{\|Pu\|_2^2}{\beta}A^+ {A^+}^Tv - \frac{\|Pu\|_2^2q^Tu}{\omega}A^+ {A^+}^T v -\frac{\beta q^Tu}{\omega}A^+u\nonumber\\
    \intertext{Collecting like terms:}
    &=\left(1 - \frac{\beta q^Tu}{\omega}\right)\left({A^+u} +\frac{\|Pu\|_2^2}{\beta}{A^+ {A^+}^Tv}\right) \nonumber\\ 
    \intertext{Using the inverse since $A^TA$ is non singular:}
    &=\left(1 - \frac{\beta q^Tu}{\omega}\right)\left({A^+u} +\frac{\|Pu\|_2^2}{\beta}({A^TA)^{-1} v}\right) \nonumber\\ 
    \intertext{Expanding $\beta q^Tu$}
     &=\left(1 - \frac{{v^TA^+{A^+}^Tv\|Pu\|_2^2} + \beta v^T{A^+}u}{\omega}\right)\left({A^+u} +\frac{\|Pu\|_2^2}{\beta}({A^TA)^{-1} v}\right) \nonumber\\ 
     \intertext{Substituting in $\omega$}
     &=\left(1 - \frac{\omega - \beta^2 + \beta v^T{A^+}u}{\omega}\right)\left({A^+u} +\frac{\|Pu\|_2^2}{\beta}({A^TA)^{-1} v}\right) \nonumber\\ 
      \intertext{Finally, rearranging we get:}
     &=\left(1 - 1-\frac{\beta v^T{A^+}u-\beta^2}{\omega}\right)\left({A^+u} +\frac{\|Pu\|_2^2}{\beta}({A^TA)^{-1} v}\right) \nonumber\\
     &=\frac{\beta}{\omega}\left({A^+u} +\frac{\|Pu\|_2^2}{\beta}({A^TA)^{-1} v}\right)\nonumber \\
     &=\frac{\beta}{\omega}{A^+u} +\frac{\|Pu\|_2^2}{\omega}({A^TA)^{-1} v}\label{first universal update}
\end{align}
\end{proof}

Setting $A = J_r$, $u = -r$ $v = -\nabla \eta$, Theorem \ref{Thrm:GN2 full rank} implies that the ``bad'' deflated Gauss--Newton step may be written as:
\begin{equation}
    \hat p^k  = \beta_1 \, p^k + \beta_2\, \left(J_r^T J_r\right)^{-1} {\nabla\eta} 
\end{equation}
with  
\begin{align}\label{eqn:beta1beta2}
       \beta_1  = \frac{\beta}{\omega} \quad\text{ and }\quad
   \beta_2  = -\frac{\|Pr\|_2^2}{\omega},
\end{align}
where 
\begin{align}
        \beta &=  1-\langle p^k, \nabla \eta\rangle , \quad P = I-J_r(x^k)J_r(x^k)^+, \quad 
        \omega= \|Pr\|_2^2\|J_r^{+T}\nabla\eta\|_2^2+\beta^2.
\end{align}

Interestingly, this deflated step contains a scalar multiple of the undeflated step, just like the deflated Newton and ``good'' Gauss--Newton methods, but also a term in the $-(J_r^TJ_r)^{-1} \nabla\eta$ direction. This is a step in a direction that reduces $\eta$  since  $(J_r^TJ_r)$ is positive definite. 

Clearly \eqref{first universal update} is  a generalization of the original deflation  update formula discussed in \cite{farrell_deflation_2020}, as the square case can be easily recovered: if we assume that $J_r$  is square, then $J_r^+ = J_r^{-1}$ so $J_rJ_r^+ = J_rJ_r^{-1} =I$ and $P = 0$, therefore $\omega = \beta^2$  and the inverse formula simplifies to the original Newton deflation update formula:
\begin{flalign}\label{square case update}
    (J_r+r\nabla\eta^T)^{-1} r=\frac{\beta}{\omega}{J_r^{-1} r} = \frac{\beta}{\beta^2}{J_r^{-1} r} =\frac{ p^k }{\beta}
\end{flalign}

It is important to note that the ``bad" deflated Gauss--Newton method does not actually solve the minimization problem given in \eqref{NLLS equation} but instead can be thought of as solving a nearby problem. This is because the ``bad" deflated Gauss--Newton  method minimizes the modified function $\mu^2f$. We therefore suggest using the modification used in the ``good" Gauss--Newton method, where the undeflated method is used when $\langle p, \nabla \eta\rangle <\epsilon$. This guarantees that the method will converge to local minima of $f$ instead of those of $\mu^2f$.

\subsubsection{Convergence and non-convergence}\label{Sec. BadGN Convergence}
Similar to the ``good" deflated Gauss--Newton method discussed above we retain the local convergence properties of the Gauss--Newton method away from any deflated points. 

Again it is important to show that the method will not converge to any deflated points.

\begin{theorem}
Let $y_1,\dots,y_n \in \mathbb R^\ell$ be a set of deflated points and $\mu(x;y_1,\dots,y_n)$ be a deflation operator as in Definition  \ref{def. Optimization Deflation Operator}. Let $\{x^k\}_{k=0}^\infty$ be the iterates of the ``bad'' deflated Gauss--Newton method (Algorithm \ref{alg. Bad Deflated Gauss--Newton}). Assume that $r$ is Lipschitz continuously differentiable, $J_r$ is uniformly full rank in a neighbourhood of $\{x^k\}_{k=0}^\infty$ and assume all steps are well-defined. Then $x^k$ will not converge to any of the deflated points $y_1,\dots,y_n$.
\end{theorem}
\begin{proof}

Assume for contradiction  that $x^k $ converges to $y\in \{y_1,\dots,y_n\}$.

Recall that the deflated step, which we will denote by $\hat p^k$, is
\begin{equation*}
        \hat p^k = \frac \beta \omega\left(p^k - \frac{r^TPr}{\beta}(J_r^TJ_r)^{-1} \nabla \eta\right) 
\end{equation*}
Recall also that the deflated step, $\hat p^k$, is taken when $\beta(x^k)<1-\epsilon$, and when $\beta(x^k)\geq 1-\epsilon$ the undeflated step is taken with a line search applied, $\alpha \, p^k$. Thus similarly to the proof of Theorem \ref{Thrm: Good GN NC} we need to consider the convergence behaviour of both of these steps. Again by the continuity of $\beta$ (as a function of $x\in\mathbb R^d\backslash \{y_1,y_2,\dots,y_n\}$), the tail of the sequence $\{x^k\}_{k=0}^\infty$ is formed by taking either all undeflated steps, $p^k$, or all deflated steps, $\hat p^k$.

First we assume that the tail of $\{x^k\}_{k=0}^\infty$ only takes undeflated steps, and the proof follows as in Theorem \ref{Thrm: Good GN NC}: as we have assumed the sequence (with a line search) converges, by Theorem \ref{thm: undeflatedGN}, we have $J_r(x^k)^T r(x^k) \to 0$. Since $J_r$ is uniformly full rank, this implies that $p^k \rightarrow 0$. By Definition \ref{def. Optimization Deflation Operator} part (2)
\begin{equation}\label{eqn: condition 2 Bad}
    \liminf_{ k \rightarrow \infty} \langle  \nabla\eta(x^k), p^k\rangle >1.
\end{equation} 
By definition of the algorithm, however, this means the algorithm takes the the deflated step in the tail, a contradiction. 

The remaining case to consider is that the tail of $\{x^k\}_{k=0}^\infty$ takes only deflated steps. Convergence of the $x^k$ implies that $\liminf_{ k \rightarrow \infty} \hat p^k=0$. Again, by the continuity of $\beta$ it must be the case that for the tail of $\{x^k\}_{k=0}^\infty$ either 
\begin{align}\label{eqn: Bad cond 0}
    &\beta(x^k)<0\qquad\text{or}\\
0<&\beta(x^k) <1-\epsilon,\label{eqn: Bad cond epsilon}
\end{align} 
which correspond to the red and yellow regions (respectively) in Figure \ref{fig: epsilon contours}. The case $\beta = 0$ yields an undefined step.
 
 First let us assume that property \eqref{eqn: Bad cond 0} is satisfied. Let us examine $\langle \hat{p}^k ,\nabla \eta\rangle$:
 \begin{align}
     \langle \hat{p}^k, \nabla \eta\rangle &= \frac{\beta}{\omega}\langle p^k, \nabla\eta\rangle - \frac{\|Pr\|^2}{\omega}\langle \nabla\eta, (J_r^T J_r)^{-1} \nabla \eta\rangle \nonumber\\
     &= \frac{\beta - \beta^2}{\omega} - \frac{\|Pr\|^2 \|J_r^{+T} \nabla \eta \|^2}{\omega} \nonumber \\
     &= \frac{\beta}{\omega}- 1.\label{eqn: beta omega}
 \end{align}
Since $\beta < 0$ and $\omega > 0$, we have that the deflated step, $\hat p^k$, is a descent direction for $\mu$. Therefore we can derive the same contradiction as in the second half of the proof of Theorem \ref{Thrm: Good GN NC}.

The remaining case to consider is when the tail of $\{x^k\}_{k=0}^\infty$ satisfies property \eqref{eqn: Bad cond epsilon}. Note that if $\beta <\omega $ then by \eqref{eqn: beta omega} we know that  $\hat p^k$, is a descent direction for $\mu$, and come to the same contradiction as above. Thus we must conclude that $\beta\geq\omega$, which leads to the following:
\begin{equation}
    \omega \leq \beta \implies \|Pr\|_2^2\|J_r^{+T}\nabla\eta\|_2^2 \leq \beta(1-\beta).
\end{equation}
From this we will deduce that $p^k \to 0$, as follows. By the definition of $\hat{p}^k$,
\begin{align}
J_r p^k = \frac{\omega}{\beta} J_r \hat{p}^k + \frac{\|Pr\|_2^2}{\beta} J_r^{+T} \nabla \eta, \label{eqn:betaomega}
\end{align}
since $J_r^{+T} = J_r (J_r^T J_r)^{-1}$. Taking norms and using the inequalities in \eqref{eqn:betaomega}, we have
\begin{align}
    \|J_r p^k\|_2 &\leq \frac{\omega} {\beta}\|J_r\hat{p}^k\|_2 + \frac{\|Pr\|_2^2 \|J_r^{+T} \nabla \eta\|_2}{\beta} \nonumber \\
    &\leq \|J_r \hat{p}^k\|_2 + \frac{1-\beta}{\|J_r^{+T} \nabla \eta\|_2} \nonumber \\
    &\leq \|J_r\|_2\left(\|\hat{p}^k\|_2 + \frac{1}{\|\nabla \eta\|_2}\right).\label{eqn:lastlinebigproof}
\end{align}
The last line follows from $\beta > 0$ and $\|\nabla \eta\|_2 = \| J_r^T J_r^{+T} \nabla \eta\|_2 \leq \|J_r\|_2\|J_r^{+T} \nabla \eta\|_2$.

The term in equation \eqref{eqn:lastlinebigproof} tends to zero because $\|J_r(x^k)\|_2$ is uniformly bounded by continuity of $J_r$, and because $\lim_{x\to y} \|\nabla \eta(x)\|_2 = \infty$ by Definition \ref{def. Optimization Deflation Operator} part (3). Because $J_r$ is uniformly full rank, $J_r p^k \to 0$ implies $p^k \to 0$. However, $p^k\rightarrow0$ implies \eqref{eqn: Bad cond 0} by Definition \ref{def. Optimization Deflation Operator} part (2), which is a contradiction.
\end{proof}
\FloatBarrier
\subsection{Gauss--Newton comparison}

In this section we have introduced the ``good" and ``bad"   deflated Gauss--Newton methods. The two big advantages of these methods over the deflated Newton method are that each step is computationally cheaper due to the fact that they do not require the calculation of the Hessian of $f$, and that they only converge to local minima, meaning that fewer deflations are required to find all local minima. 

There are also advantages and disadvantages between the two different deflated Gauss--Newton methods we have discussed in this section. We note that each step of the ``good'' method is cheaper than the ``bad'' method, but it is possible that the extra term in the deflation step means that the method may require fewer iterations to converge. In general we have not found that one method or the other is consistently better than the other.

\section{Implementation and experiments}\label{sec: Application}
We will first discuss a couple of two dimensional examples, then move onto some high dimensional problems, and finally we will investigate a physical example coming from an Inelastic Neutron Scattering experiment. We will cover a range of  zero and non-zero residual problems. 

The deflation operator defined in \eqref{eqn. Deflation Operator} is used in all of the  following experiments, with multiple deflations applied as in \eqref{eqn. Multi Shift Operator} and parameters  fixed as: $\theta =2, \sigma=1$  and $\epsilon = 0.01$ (discussed in Figure \ref{fig: epsilon contours}). We have tried using different deflation operators but did not find any meaningful differences worth reporting. We used a quadratic line search whenever a line search was applied. The least squares problems are solved using \texttt{lsqminnorm} for the ``good'' method and by a QR factorization for the ``bad'' method. This QR factorization is recycled to compute the extra terms in the deflated Gauss--Newton method. Open source MATLAB code for all experiments can be found at \url{https://github.com/AlbanBloorRiley/DeflatedGaussNewton}.

\subsection{A test problem with many local minima}
This is a nonlinear least squares problem in two variables with 42 local minima (143 stationary points), defined using a truncated product expansion in $\cos$ and $\sin$: 
\begin{align}
    r(x,y) = \begin{pmatrix}\displaystyle
        a \prod^3_{k=1}1-\frac{(x+y)^2}{k^2\pi^2}\\\displaystyle
        a\prod^3_{k=1}1-\frac{(x-y)^2}{(k-1/2)^2\pi^2}\\
        a+ 0.01(x^2+y^2)
\end{pmatrix}\label{FTrig}
\end{align}
The parameter $a$ defines defines the height of the objective function, which can be used to test how the various methods perform with potentially large residual problems. In all the experiments here $a=10$. As seen in Figure \ref{fig: FTrig} the Newton method requires 143 deflations to find all of the stationary points, and so guarantee finding all the local minima. Both the ``good'' and ``bad''  Gauss--Newton methods find all the local minima in just 42 deflations. Only the plot from the ``good'' Gauss--Newton method is shown as the result is almost identical to the one produced by the ``bad'' method.  
   
\begin{figure}[!ht]
    \centering
    \includegraphics[width =0.8\textwidth]{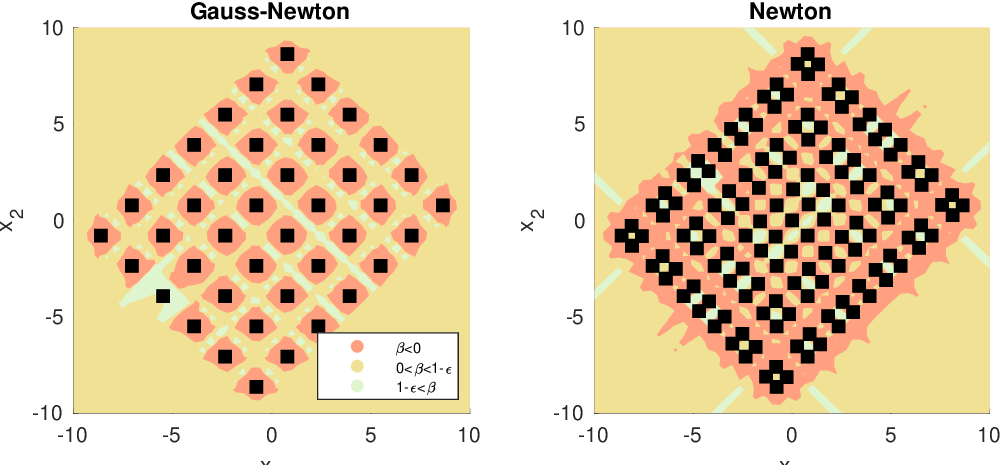}
    \caption{Multiple minima found to the nonlinear least squares problem given by \eqref{FTrig}. Left: all 42 minima found by the ``good'' Gauss--Newton algorithm. Right: the 143 stationary points found by using the Newton method for optimization algorithm. Note that in the green region $1-\epsilon <\beta$, in the yellow region $0<\beta<1-\epsilon$, and in the red region $\beta<0$, for $\epsilon = 0.01$.}    \label{fig: FTrig}
\end{figure}
We  compare the convergence rates of all three methods discussed in this paper in Figure \ref{fig: FTrig Compare}, clearly showing that all three methods have a quadratic local convergence rate after period of non-convergence. It also shows that the number of iterations required for each deflation can vary wildly.

\begin{figure}[!ht]
    \centering
    \includegraphics[width =0.8\textwidth]{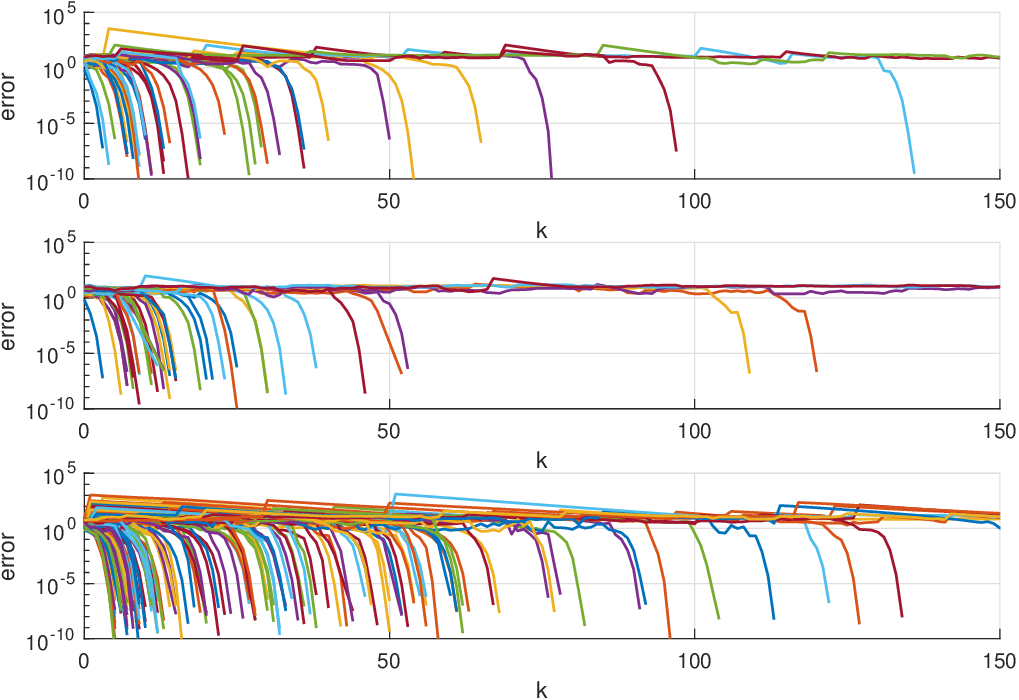}
    \caption{Comparison of the convergence behaviour for computing all solutions to the nonlinear least squares problem given by \eqref{FTrig}. Top: ``good'' Gauss–Newton (42 local minima). Middle: ``bad'' Gauss--Newton (42 local minima). Bottom: Newton for Optimization (143 stationary points). All methods started from an initial guess of \texttt{[1;3]}.}
    \label{fig: FTrig Compare}
\end{figure}

We may also use this example to compare with other global minimization methods, such as Matlab's `MultiStart'  algorithm, from the global minimization toolbox \cite{inc_global_2024}.

Using MultiStart with the `lsqnonlin' solver on the unconstrained problem only finds a handful of minima, even when 10000 start points are used.  When the method is constrained to the 20 by 20 square, centred about the origin, MultiStart finds all 42 minima consistently with only 300 start points; when constrained to the 40 by 40 square however not all 42 minima are consistently found. The bounded MultiStart method still requires many more functions evaluations than the Deflated Gauss--Newton methods as can be seen in Figure \ref{fig: MultiStart Compare}, even when lsqnonlin is set up to use the Jacobian not just use a finite difference approximation. The figure also shows how the Gauss--Newton methods are  faster with respect to computation time. 

\begin{figure}[!ht]
    \centering
    \includegraphics[width =0.9\textwidth]{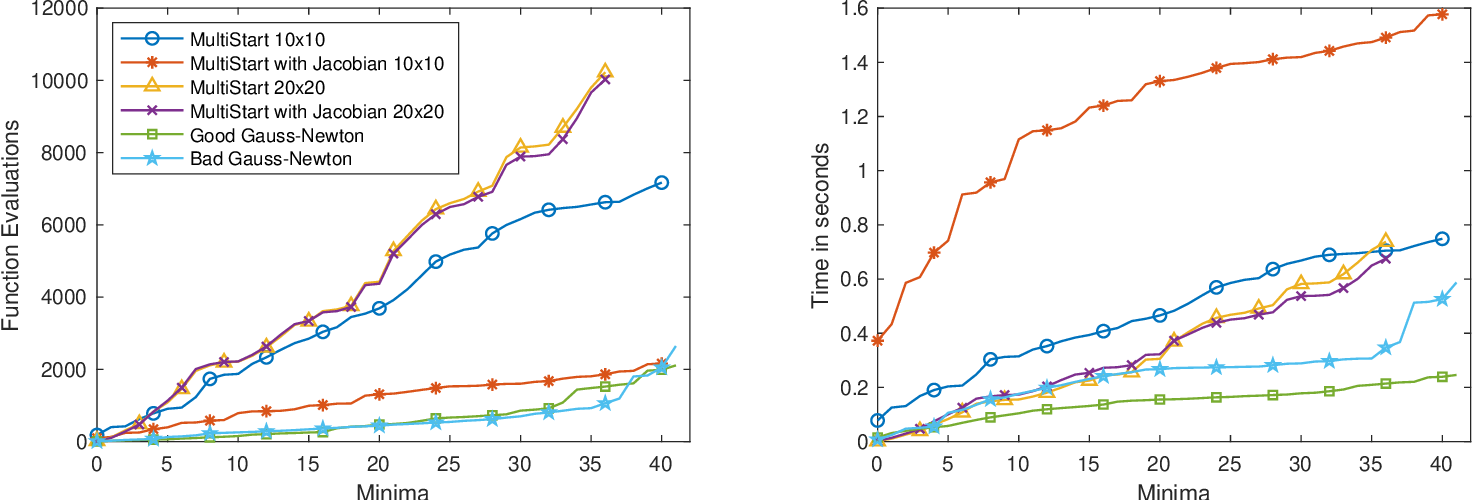}
    \caption{ Left: Cumulative number of function evaluations required for each method to find all 42 local minima. Right: cumulative time to compute the local minima on a 2021 M1 Macbook Pro. For the multistart methods 300 initial points are used, with function evaluations/timings including all the points tried up until the 42nd local minimum is found.}
    \label{fig: MultiStart Compare}
\end{figure}

We have performed experiments to investigate the effect of the value of $\epsilon$ on the convergence of the ``good'' and ``bad'' deflated Gauss--Newton methods. In Figure \ref{fig: epsilon convergence comparison}, we observe that the convergence rate of ``good'' deflated Gauss--Newton is largely unaffected by varying $\epsilon$, whereas the convergence  of the ``Bad'' deflated Gauss--Newton method is greatly affected. The bottom-left plot shows that even though the ``Bad'' method converges, all solutions it finds after deflation have a large value of $\|\nabla f(x^k)\|$, signifying that the solutions are not actually minima of $f$. As discussed in section \ref{Sec. BadGN Convergence} this is because when $\epsilon =0$ the ``bad'' method converges to local minima of $\mu^2f$ not $f$, indeed this is one reason we call it ``bad''. However Figure \ref{fig: epsilon convergence comparison} also shows that when $\epsilon>0$ the method does converge to minima of $f$, as undeflated steps (the lines in bold) are taken when close to a minimum.

\begin{figure}[!ht]
    \centering
    \includegraphics[width =0.8\textwidth]{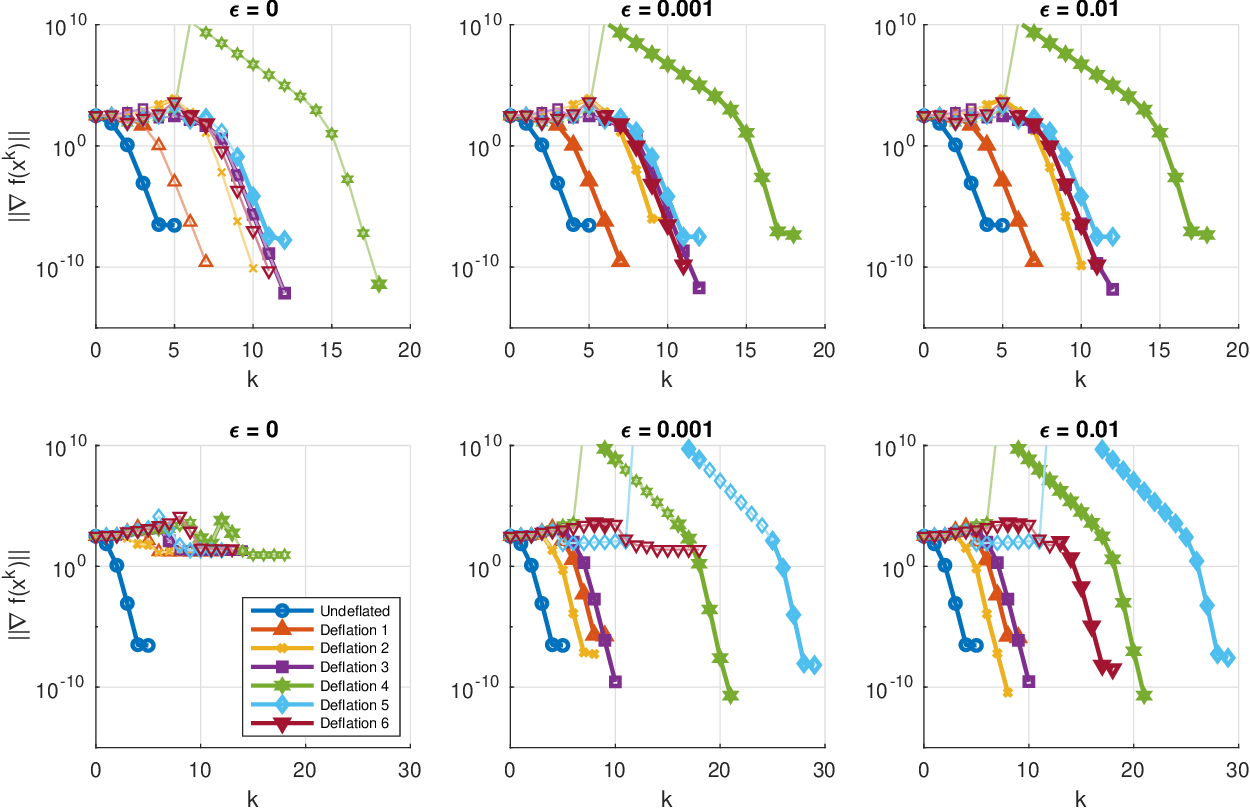}
    \caption{Comparison of the convergence behaviour of the two deflated Gauss--Newton methods, and the effect of different values of  $\epsilon $ on the convergence rate, for the first 7 minima of function \ref{FTrig}. Top: ``Good'' Gauss--Newton. Bottom: ``Bad'' Gauss--Newton. The bold (thinker) lines indicate the steps where undeflated Gauss--Newton steps are taken, and the not bold lines indicate where  the deflated steps are taken. $\nabla f(x^k)$ is as defined in \eqref{eqn:gradf}.}
    \label{fig: epsilon convergence comparison}
\end{figure}

\subsection{Multiple solutions of nonlinear BVPs by Fourier extension}\label{subsec FE}

Fourier extension is a technique to approximate functions on an arbitrary bounded set $\Omega \subset \mathbb{R}^d$ using Fourier series that are periodic on a box $B = [a_1,b_1] \times \cdots \times [a_d,b_d]$ containing $\Omega$. We stick to the case $d = 1$, $\Omega = [0,1]$ and $B = [-1,1]$ in what follows. For a smooth function $f : [0,1] \to \mathbb{R}$, a Fourier extension approximation is given by
\begin{equation}
    f_N(x) = \sum_{j=-n}^n c_j \mathrm{e}^{\mathrm{i} j \pi x},
\end{equation}
where $N = 2n+1$. Note that $f_N$ is periodic on $[-1,1]$. 

Fourier extensions inherit several nice properties from the Fourier basis: we can find the derivative of $f_N$ by simply multiplying its coefficients by $\mathrm{i}j\pi$, we can evaluate $f_N(x_k)$ on an equispaced grid 
\begin{equation}\label{eqn:FEpts}
\{x_k = k/m : k = 0,\ldots,m\},
\end{equation}
where $m \geq 2n$ in $\mathcal{O}(m\log m)$ operations by the Inverse Fast Fourier Transform (IFFT), they have spectral approximation properties \cite{webb_pointwise_2020}, and we can compute convolutions in $\mathcal{O}(N \log N)$ operations \cite{xu_fast_2017}. However, computation of the coefficients $c_j$ is extremely ill-conditioned because there exist nonzero coefficients such that $f_N$ approximates 0 on $[0,1]$ while its values on $[-1,0]$ are unconstrained. Despite this ill-conditioning, adequate coefficients $c_j$ can be computed stably by solving an oversampled least squares collocation problem at equally spaced points in $[0,1]$, \cite{adcock_frames_2019, adcock_frames_2020, matthysen_fast_2016, matthysen_function_2018}. 

In the following examples we discretize and solve nonlinear Boundary Value Problems (BVPs) as oversampled nonlinear least squares collocation problems at equally spaced points in $[0,1]$. We will give two examples to show that the ``good'' and ``bad'' deflated Gauss--Newton methods can be effective in the cases where there are multiple isolated solutions to a nonlinear boundary value problem.

There are two ways in which the resulting optimization problems are more complicated than those presented so far in the paper. First, the solution is a complex-valued vector $\mathbf{c}$, which might appear to cause problems for the condition $\langle \nabla \eta, p^k\rangle>\epsilon$ in our algorithms. However, the analogous condition in complex arithmetic for measuring to what extent $p^k$ is a direction of ascent for $\eta$, is $\mathrm{Re}\langle\nabla\eta,p^k\rangle>\epsilon$. Second, the norm used to measure distance in the deflation operator should not be the $2$-norm of $\mathbf{c}$. This is because there are many different coefficient vectors that yield approximately the same function, so deflation measuring distance between coefficient vectors will not be particularly effective. We instead use the norm
\begin{equation}
    \| \mathbf{c} \|_{\mathrm{FE}} := \left(\frac1{m+1}\sum_{k=0}^m \left|\sum_{j=-n}^n c_j \mathrm{e}^{\mathrm{i} j \pi x_k}\right|^2 \right)^{1/2},
\end{equation}
which is the 2-norm of the function values on the equispaced grid \eqref{eqn:FEpts}. This means that the deflation process measures distance between coefficients by measuring the difference between the associated function values.

\subsubsection{Bratu equation}
The Bratu equation is given by \cite[Chapter 16]{trefethen_exploring_2017}:

\begin{align}\label{eqn:bratu}
    u'' + 3\exp(u) = 0,\qquad u(0) = 0,\qquad u(1) = 0
\end{align}

After discretization, we minimize the following nonlinear least squares residual for $\mathbf{c} \in \mathbb{C}^{N}$:
\begin{equation}
\begin{split}
    f(\mathbf{c}) =& \frac1{m+1}\frac12\sum_{k=0}^m \left|\sum_{j=-n}^n -j^2\pi^2 c_j \mathrm{e}^{\mathrm{i} j \pi x_k} + 3 \exp\left(\sum_{j=-n}^n c_j \mathrm{e}^{\mathrm{i} j \pi x_k}\right)\right|^2 \\
    &  \qquad +\, \frac12\left|\sum_{j=-n}^n c_j\right|^2 \, + \, \frac12\left|\sum_{j=-n}^n c_j \, \mathrm{(-1)}^j\right|^2. \label{eqn:bratunls}
    \end{split}
\end{equation}
The first sum corresponds to enforcing the ODE at the collocation points $x_k$ (defined in equation \eqref{eqn:FEpts}) and the last two terms correspond to enforcing the two boundary conditions.

In the experiment depicted in Figure \ref{fig: Bratu}, we set $n = 100$, $m = 400$, so that there are 403 least squares constraints and 201 unknowns. The initial guess is the zero function, both initially, and each time we deflate. 
\begin{figure}[!ht]
    \centering
    \includegraphics[width =0.8\textwidth]{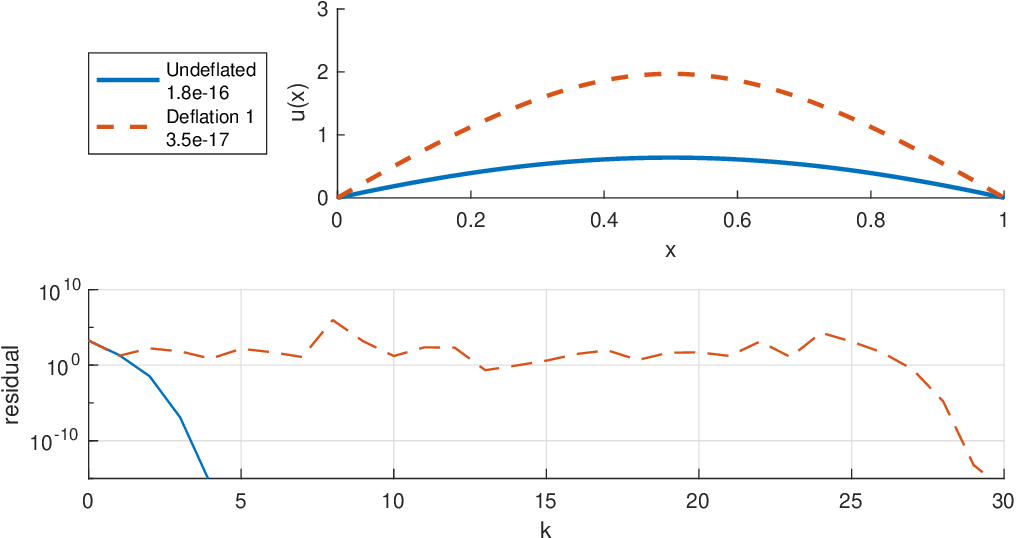}
    \caption{Multiple solutions of the Bratu equation \eqref{eqn:bratu}, computed by solving the nonlinear least squares problem \eqref{eqn:bratunls}. Top: plots of the functions that the ``good'' Deflated Gauss--Newton algorithm found. Bottom: the residual against each iteration.}
    \label{fig: Bratu}
\end{figure}

The convergence behaviour seen in Figure \ref{fig: Bratu} is typical, and the analogous plot using the Bad deflated Gauss--Newton method is almost the same. What we see for the deflated iteration is that there is a long period of non-convergence, before finding a basin of attraction and converging at the same rate as the undeflated Gauss--Newton method in that basin. One reason why this non-convergence happens is because the initial guess for the iteration is set to the be the same as it was for the first undeflated iteration (which in this case is zero). The local convergence is the same as for the undeflated Gauss--Newton method by design --- at a new local minima $x$ we must have $\langle \nabla \eta(x), p(x)\rangle = 0$, so in a neighbourhood of the new local minima we must have $\langle \nabla \eta(x), p(x) \rangle \leq \epsilon$, so the undeflated step is taken.

\subsubsection{Carrier equation}\label{Ex. Carrier}
The Carrier equation \cite[Chapter 16]{trefethen_exploring_2017}, is
\begin{align}\label{eqn:carrier}
    0.05 \, u'' + 8x(1-x)u + u^2 = 1,\qquad u(0) = 0,\qquad u(1) = 0.
\end{align}
Similarly to the Bratu example, we minimize the following nonlinear least squares residual for $\mathbf{c} \in \mathbb{C}^{N}$:
\begin{equation}
\begin{split}
    f(\mathbf{c}) =& \frac1{m+1}\frac12\sum_{k=0}^m \left|\sum_{j=-n}^n \left(-0.05 \, j^2\pi^2 + 8 x_k(1-x_k)\right)c_j \mathrm{e}^{\mathrm{i} j \pi x_k} + \left(\sum_{j=-n}^n c_j \mathrm{e}^{\mathrm{i} j \pi x_k}\right)^2 - 1\right|^2 \\
    & \qquad \quad + \, \frac12\left|\sum_{j=-n}^n c_j\right|^2 \, + \, \frac12\left|\sum_{j=-n}^n c_j \, \mathrm{(-1)}^j\right|^2.\label{eqn:carriernls}
\end{split}
\end{equation}

In the experiment depicted in Figure \ref{fig: Carrier}, we set $n = 100$, $m = 400$, so that there are 403 least squares constraints and 201 unknowns. The initial guess is the zero function, both initially, and each time we deflate. 

\begin{figure}[!ht]
    \centering
    \includegraphics[width =0.8\textwidth]{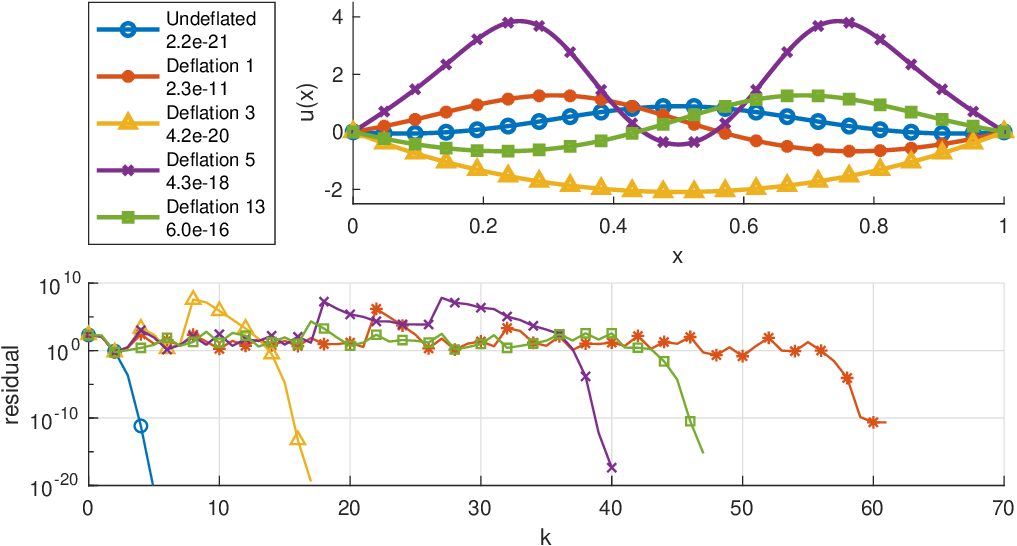}
    \caption{Multiple solutions of the carrier equation \eqref{eqn:carrier}, computed by solving the nonlinear least squares problem \eqref{eqn:carriernls}. Top: plots of the functions that the ``good'' Deflated Gauss--Newton algorithm found. Bottom: the residual against each iteration.}
    \label{fig: Carrier}
\end{figure}

Interestingly, the fourth solution found in Figure \ref{fig: Carrier} is distinct from the solutions found in \cite{trefethen_exploring_2017}. The solution persists when we increase the degrees of freedom in the approximation, and the residual is very small, so we conclude that this is likely a proper solution of the differential equation that was missed, and not a feature of the discretization.
 
\begin{figure}[!ht]
    \centering
    \includegraphics[width =0.8\textwidth]{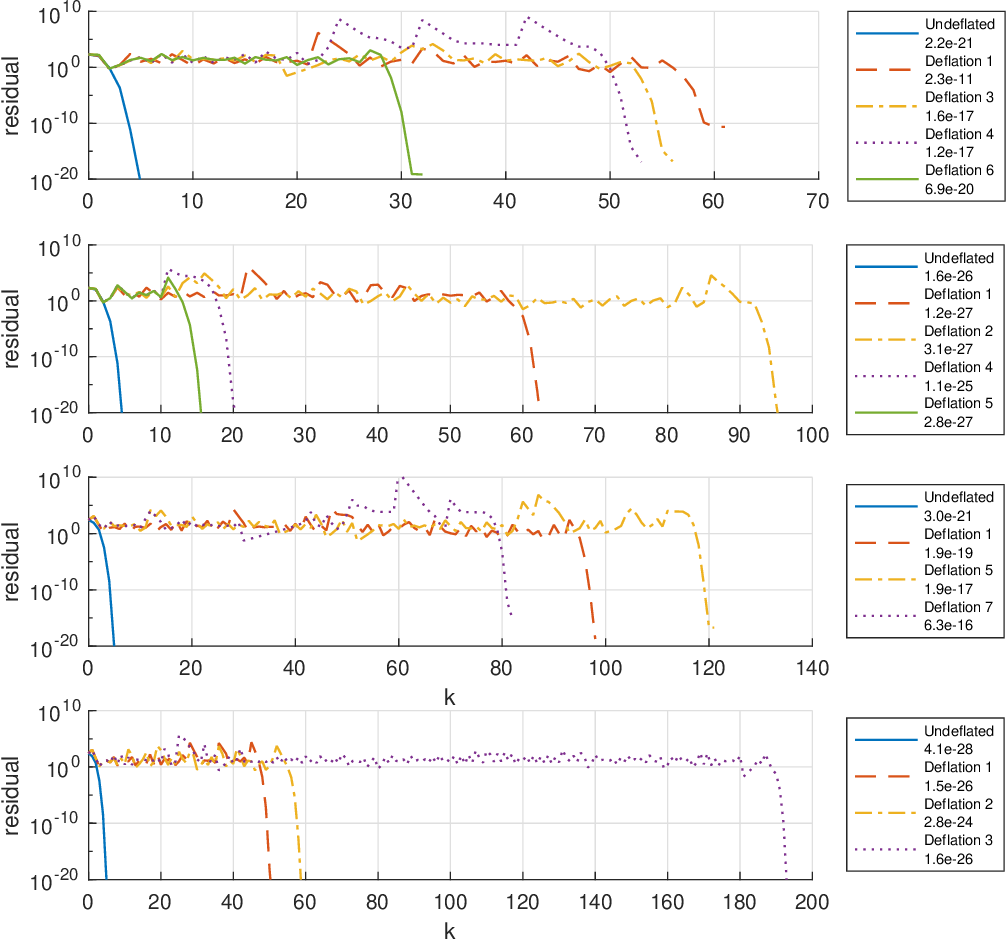}
    \caption{Comparison of the convergence behaviour for computing multiple solutions to the Carrier equation. Top: ``good" Gauss--Newton with initial guess $\mathbf{c} = \mathbf{0}$. Second: ``bad'' Gauss--Newton with initial guess $\mathbf{c} = \mathbf{0}$. Third: ``good'' Gauss--Newton with initial guess $u(x) = x(1-x)$. Bottom: ``bad'' Gauss--Newton with initial guess $u(x) = x(1-x)$. We do not mean to imply any value judgement on ``good'' or ``bad'' deflation techniques' convergence properties for this problem, merely that the initial period of non-convergence can vary wildly in length, and in unpredictable ways.}
    \label{fig: Carrier Compare}
\end{figure}
\subsection{Inverse eigenvalue problems}
One application where the deflated methods outlined in this paper are applicable is to the solution of inverse eigenvalue problems (IEP) that arise from spectroscopic techniques, such as inelastic neutron scattering (INS) experiments \cite{lovesey_theory_2003}. INS can be used to accurately measure magnetic excitations in materials that possess interacting electron spins. When studying finite spin systems, such as single ions or molecular-based magnets, INS is used to experimentally quantify the energy between quantum spin states that relate directly to eigenvalues of a Hamiltonian that describes the quantum spin dynamics of the compound \cite{RevModPhys.85.367, baker_neutron_2012, baker_spectroscopy_2014}. Such information is of both fundamental and technological importance for further understanding quantum phenomena and how electronic quantum spins could be implemented for novel quantum applications such as information processes and sensing. Accurate quantification of the quantum spin dynamics of magnetic molecules following an INS experiment requires determining parameters associated with an appropriate spin Hamiltonian model. The spin Hamiltonian operator that gives rise to these experimentally determined eigenvalues is to be computed by fitting, such as nonlinear least squares.

\begin{definition}[Inverse Eigenvalue Problem \cite{chu_inverse_2005}]
Given complex numbers $\lambda_1,\dots,\lambda_n$ and a basis of matrices $A_0,\dots,A_\ell \in \mathbb C^{n\times n}$, let 
\begin{equation}\label{IEP}
    r( x) = 
    \begin{pmatrix}
     \lambda_1( x) - \lambda_1 \\
    \vdots \\
    \lambda_{n} ( x) - \lambda_{n} 
    \end{pmatrix} .
\end{equation}
and   $\lambda_1(x),\dots\lambda_{n} (x)$  be the eigenvalues of the matrix 
\begin{equation}
A( x) = A_0 + \sum^\ell_{i=1} x_i A_i.
\end{equation}
Then the inverse eigenvalue problem is to find the parameters $x \in \mathbb R^\ell$ that minimize
\begin{equation}
    f( x) = 
    \frac 1 2 ||r( x)||^2_2.
\end{equation}
\end{definition}

For the INS example that we will consider, the $\lambda_1,\dots,\lambda_n$ are the experimental eigenvalues and calculated from the experimental data and the basis matrices are  Stevens operators\footnote{See \cite{gatteschi_molecular_2006} and \cite{rudowicz_generalization_2004} for more general details on Stevens operators.}. The operators are all calculated using the easyspin MATLAB package \cite{stoll_easyspin_2006}. 
\subsubsection{Mn12}
The selected example concerns the INS spectrum of Manganese-12-acetate. This molecule gained importance following the identification that it acts like a nano-sized magnet with a molecular magnetic coercivity and the identification of quantum tunnelling of magnetisation see \cite{friedman_macroscopic_1996, sessoli_magnetic_1993}. The INS spectrum of Manganese-12-acetate was measured in order to obtain a precise description of an appropriate Spin Hamiltonian model that accurately describes its magnetic properties. The Hamiltonian of the system is a $21\times21$ matrix  and and can be modelled using 4 basis matrices \cite{bircher_transverse_2004}:
\begin{equation}\label{Ham}
    A(B_2^0, B_4^0, B_2^2, B_4^4) =
    B^0_2O^0_2 + B^0_4O^0_4 +B^2_2O^2_2 +B^4_4O^4_4 \in \mathbb R^{21\times 21}
\end{equation}

Where the $B$s are the parameters to be found and the $O$s are the Stevens operators are defined, in this case with a spin of $S = 10$, as: 
\begin{align*}
O_2^0 &= 3S_z^2 - X\\
O_2^2 &= \frac 1 2 (S_+^2+ S_-^2)\\
O_4^0 &= 35S_z^4 - (30X-25)S^2_z +(3X^2-6X)\\
O_4^4 &= \frac 1 2 (S_+^4+ S_-^4)
\end{align*}
where $X = S(S+1)I\in \mathbb R^{S(S+1)\times S(S+1)}$, and
\begin{align*}
    (S_z)_{k,k} \quad&= (S + 1-k)\\
    (S_+)_{k,k+1} &=  \sqrt{k(2S+1-k)}\\
    (S_-)_{k+1,k} &=  \sqrt{k(2S+1-k)}
\end{align*}

 Four distinct parameter sets are given as solutions to this problem, as found in \cite{bircher_transverse_2004}. Since we do not have access to the original experimental data  the eigenvalues used were reverse engineered from the  solutions given in \cite{bircher_transverse_2004}. The experiment depicted in Figure \ref{fig: Mn12} shows the convergence rates of the three different methods applied to the inverse eigenvalue problem \eqref{IEP}.

\begin{figure}[!ht]
    \centering
    \includegraphics[width =0.8\textwidth]{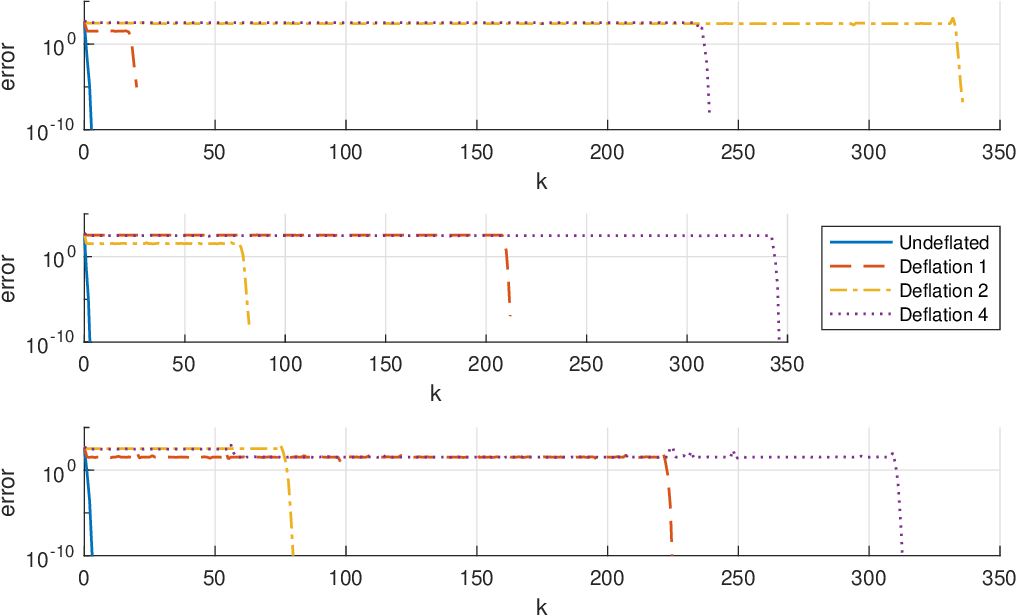}
    \caption{Comparison of the convergence behaviour for computing all four solutions to the Mn$_{12}$ inverse eigenvalue problem. Top: ``good'' Gauss–Newton. Middle: ``bad'' Gauss--Newton. Bottom:
Newton for Optimization.}
    \label{fig: Mn12}
\end{figure}

\section{Conclusions}
We have presented two new deflated optimization algorithms, both based on the Gauss--Newton method. Both algorithms find local minima, without converging to other stationary points, thus reducing the number of deflations needed compared to the deflated Newton method applied to the first order optimality conditions. They also do not require the calculation of a Hessian matrix, allowing the application of the method to a wider range of problems. We then showed how these methods can be effective for problems with a high dimension (up to 400) or a high number of local minima  (up to 42). We limited ourselves to cases where $J_r$ is full rank, a constraint which could be lifted in future work. It is not immediately apparent which of the ``good'' or ``bad'' methods is best, but both have been shown to have locally quadratic convergence to local minima comparable to the undeflated Gauss--Newton method (under similar assumptions). Lastly, we have shown that these methods can be used to solve least squares inverse eigenvalue problems arising from Inelastic Neutron Scattering experiments.

There are several aspects of deflation techniques not explored in this paper. For instance, we considered only a single type of deflation operator (from equations \eqref{eqn. Deflation Operator} and \eqref{eqn. Multi Shift Operator}) and we always chose the subsequent initial guesses to be the same as that used for the first, undeflated iteration. On the latter point, it could be argued that starting from the same initial guess could lead to the initial period of non-convergence being prolonged because the iterates revisit parts of the solution space that were visited in early deflations before deflation ``kicks in''. However, in Figure \ref{fig: Simple Beta Contour} we see that deflation reverses the direction of the first step after two deflations. It could be possible to develop a heuristic strategy for choosing the subsequent initial guesses in a way that reduces the period of non-convergence. Alternatively, an approach similar to MultiStart, but with deflation seems promising.

\section*{Acknowledgements}
We are grateful to Nick Higham, Jonas Latz and Fran\c{c}oise Tisseur for feedback on early drafts of this paper. We also thank the reviewers for their insightful comments that improved this paper greatly. MW thanks the Polish National Science Centre (SONATA-BIS-9), project no. 2019/34/E/ST1/00390, for the funding that supported some of this research. ABR thanks the University of Manchester for a Dean's  Doctoral Scholarship.

\FloatBarrier
\bibliographystyle{siamplain}
\bibliography{references}
\begin{appendix}
\section{Proof of Theorem \ref{Thrm: Deflated Newton}}\label{Deflated Newton proof}
\begin{proof}
The Newton step for the deflated problem, $\hat{p}^k$, is given by
\begin{equation}
\hat{p}^k = - J_{\mu r}(x^k)^{-1}\mu(x^k)r(x^k).
\end{equation}
The Jacobian of the deflated system can be expanded using the product rule,
\begin{equation}
    J_{\mu r}(x) =   \mu(x) J_r(x) + r(x) {\nabla\mu(x)}^T.
\end{equation}
For ease of notation, let us write
$$
\mu(x)^{-1} J_{\mu r}(x) = J_r(x) + r(x) \nabla \eta (x)^T,
$$
where $\eta(x) = \log(\mu(x))$, so $\nabla \eta = \mu^{-1} \nabla \mu$. The Sherman--Morrison formula implies
\begin{equation}
 J_{\mu r}(x)^{-1}\mu(x) = J_r(x)^{-1} - \frac{J_r(x)^{-1} r(x) \nabla \eta(x)^T J_r(x)^{-1}}{1 + \nabla \eta(x)^T J_r(x)^{-1}r(x)}.
\end{equation}
Using the undeflated step $p^k = -J_r(x^k)^{-1} r(x_k)$, we can write
\begin{equation}
 J_{\mu r}(x^k)^{-1}\mu(x^k) = J_r(x^k)^{-1} + \frac{p^k \nabla \eta(x^k)^T J_r(x^k)^{-1}}{1 - \langle \nabla \eta(x), p^k\rangle}.
\end{equation}
Now we can substitute this into the formula for $\hat{p}^k$, and use the formula for $p^k$ to obtain
\begin{align}
\hat{p}^k &= -J_r(x^k)^{-1}r(x^k) - \frac{p^k \nabla \eta(x^k)^T J_r(x^k)^{-1}r(x^k)}{1 - \langle \nabla \eta(x^k), p^k\rangle} = (1 - \langle \nabla \eta(x^k), p^k\rangle)^{-1} p^k,
\end{align}
which is equivalent to the deflated Newton step given in Algorithm \ref{alg. Deflated Newton}.
\end{proof}
\end{appendix}
\end{document}